\documentclass[reqno]{amsart}
\usepackage{mathrsfs}
\usepackage{txfonts}
\usepackage{amsfonts}
\usepackage{amssymb}
\usepackage[mathscr]{eucal}
\usepackage{amsmath}
\usepackage[bookmarksopen,colorlinks,citecolor=blue,
linkcolor=black,pdfstartview=FitH]{hyperref}
\usepackage{ytableau}
\usepackage{float}
\usepackage{url}
\usepackage{appendix}
\usepackage{algorithm}
\usepackage{algorithmic}
\usepackage[all]{xy}

\newtheorem{theorem}{Theorem}[section]

\newtheorem{lemma}[theorem]{Lemma}
\newtheorem{proposition}[theorem]{Proposition}
\theoremstyle{definition}
\newtheorem{definition}[theorem]{Definition}

\newtheorem{example}[theorem]{Example}

\newtheorem{question}[theorem]{Question}
\newtheorem{conjecture}{Conjecture}
\numberwithin{equation}{section}
\theoremstyle{remark}
\newtheorem{remark}[theorem]{Remark}

\makeatletter
\@namedef{subjclassname@2020}{\textup{2020} Mathematics Subject Classification}
\makeatother

\allowdisplaybreaks[4]

\begin{document}

\title{Deflation conjecture and local dimensions of Brent equations}
\pdfbookmark[0]{Deflation conjecture and local dimensions of Brent equations}{}

\author[li]{Xin Li}
\address{Department of Mathematics, Zhejiang University of Technology, Hangzhou 310023, P. R. China}
\email{xinli1019@126.com (Xin Li)}
%
%
\author[zhang]{Liping Zhang}
\address{School of Mathematical Sciences, Qufu Normal University, Qufu 273165, P. R. China}
\email{zhanglp06@gmail.com (Liping Zhang)}
\author[ke]{Yifen Ke}
\address{School of Mathematics and Statistics, Fujian Normal University, 350117 Fuzhou, P.R. China}
\email{keyifen@fjnu.edu.cn (Yifen Ke)}





\subjclass[2020]{Primary 14Q15; Secondary 65H10}
\keywords{Deflation, Singularity, Brent equations, Local dimension,  Jacobian}

\begin{abstract}
In this paper, a classical deflation process raised by Dayton, Li and Zeng is realized for the Brent equations, which provides new  bounds for local dimensions of the solution set. Originally, this deflation process focuses on isolated solutions. We generalize it to the case of irreducible components and a related conjecture is given. We analyze its realization  and apply it to the Brent equations. The decrease of the nullities is easily observed. So the deflation process can be served as a useful tool for determining the local dimensions.  In addition, our result implies that along with the decrease of the tensor rank, the singular solutions will become more and more.
\end{abstract}
\maketitle

\section{Introduction}\label{sec:intro}

This paper is about study of certain algebraic varieties, related to the problem of fast matrix multiplication. For given $m, n, p$, and $r$, the set of all bilinear schemes, of bilinear complexity $r$, for multiplication of an $m \times n$ matrix by an $n \times p$ matrix may be considered as an affine algebraic variety. The defining equations of this variety are known as Brent equations. In \cite[Sect. 8]{Heule21}, Heule et al. raised the question what is the dimension of this variety, especially in the case $m=n=p=3$, $r=23$. The goal of the present article is to study local dimensions of Brent varieties at some points.
The tool used here is the deflation process described in \cite{DZeng-05,DLZeng-11}. It improves the result in our recent paper \cite{LBZ-23}.


Let $F(\textbf{x}):~\mathbb{C}^n\to\mathbb{C}^m$ be a polynomial system with solution set $\textbf{V}(F)$. Let $J(\textbf{x})\in \mathbb{C}^{m\times n}$ be the Jacobian of $F(\textbf{x})$.
Given a solution $\hat{\textbf{x}}\in \textbf{V}(F)$ and let $\dim_{\hat{\textbf{x}}} \textbf{V}(F)$ denote its local dimension. The Jacobian criterion of \cite[Theorem 4.1.12]{DS07} tells us that the nullity of $J(\hat{\textbf{x}})$ is an upper bound of $\dim_{\hat{\textbf{x}}}\textbf{V}(F)$:
\begin{equation}\label{eq-nullj}
\dim_{\hat{\textbf{x}}} \textbf{V}(F)\leq n-\operatorname{rank}~\left( J(\hat{\textbf{x}})\right).
\end{equation}
Often, (\ref{eq-nullj}) may be not an equality, that is,
\begin{equation}\label{eq-nulljs}
\dim_{\hat{\textbf{x}}} \textbf{V}(F)< n-\operatorname{rank}~\left( J(\hat{\textbf{x}})\right).
\end{equation}
When (\ref{eq-nulljs}) happens, in numerical algebraic geometry, the solution $\hat{\textbf{x}}$ is said to be singular (or ultrasingular) \cite{BHSW-13,SW-05,Zeng-23}. It is well-known that the locally quadratical convergence of Newton's method is only  guaranteed when (\ref{eq-nullj}) is an equality (see e.g. \cite{Zeng-23}).
Near the neighbourhood of singular solutions, Newton's method often loses its locally quadratical convergence. In order to recover Newton's method and turn (\ref{eq-nullj}) into an equality, the ``deflation'' technique is introduced, which is also a useful tool for the estimation of local dimensions.

The deflation technique was firstly introduced by Ojika \cite{OWM-83,O-87} for the isolated solution. After that many improvements and generalizations were given \cite{DZeng-05,DLZeng-11,HW-13,LVZ-06}. Given a polynomial system and a solution, a deflation process produces a new expanded polynomial system and a new expanded solution. For a generic (smooth) point, the nullity of Jacobian of the new expanded solution can be much closer to its local dimension. For isolated singular solutions, it was shown in \cite{DZeng-05,DLZeng-11,LVZ-06} that the nullity will decrease to zero after a finite number of steps. For a positive dimensional irreducible component, Hauenstein and Wampler introduced a new form of deflation called strong deflation \cite{HW-13}, and they refer to the deflation process raised in \cite{DZeng-05,DLZeng-11,LVZ-06} as weak deflation.
Through strong deflation process, Hauenstein and Wampler  proved that their deflation sequence will decrease to the local dimension of the generic (smooth) points of the irreducible component \cite[Lemma 4.7]{HW-13}. For the other (singular) points, the stabilized number of the (strong) deflation sequence will be no greater than their local dimension \cite[Sect.5.3]{HW-13}.

For the (weak) deflation raised in \cite{DZeng-05,DLZeng-11,LVZ-06}, it focuses on the isolated solutions. However, it is not hard to see that it is also suitable to the positive dimensional solutions. In \cite[Sect. 9]{Zeng-23}, in order to recover the quadratic convergence rate of Newton's iteration at ultrasingular zeroes, Zeng conjectured that for generic (smooth) points, the weak deflation process will also terminate at the local dimension of positive dimensional solutions after a finite number of steps. That is, similar properties of strong deflation in \cite{HW-13} also hold for the weak deflation process raised in \cite{DZeng-05,DLZeng-11} under the positive dimensional generalization.

By Lemma 4.7 of \cite{HW-13}, the strong deflation  can turn (\ref{eq-nullj}) into an equality. However, it will introduce much more new variables and polynomials than the weak deflation.
So the implementation of strong deflation is not so practical for large polynomial systems, such as the Brent equations. In addition, the local dimension determination methods raised in \cite{Bates09,Lai-21,WHS-11} are also not suitable to large polynomial systems, too. Let $\langle m,n,p\rangle $ denote the matrix multiplication tensor. We find that it is convenient to calculate the first 4 numbers of the weak deflation sequences for $mnp\leq 125$. So our results imply that the weak deflation is a useful tool for determining the local dimensions for polynomial systems arising in the field of bilinear complexity \cite{Lands-17}.

\subsection*{Organization and main results}
Firstly, after some preliminaries, more details of the conjecture in \cite{Zeng-23} are given,  see Section \ref{sec-prelim} and Conjecture \ref{conj-dc} in Section \ref{sec:defl}. Some elementary results are discussed, which are similar to those in \cite{HW-13}. Secondly, the constructions of the Jacobians in the first three steps of the deflation process in \cite{DZeng-05,DLZeng-11} are analysed, which are also suitable for other steps, see Section \ref{se-realization}. Thirdly, numerical experiments are given for the solutions of the Brent equations, see Section \ref{se-numerexp}. The first 4 numbers of the deflation sequences of many known solutions (see e.g. \cite{Faw22+,Heule19}) are calculated, which provide bounds for their local dimensions.
Some interesting results are found. For example, in the sense of solving polynomial systems, the local dimension of Strassen's classical solution is 23, see Section \ref{secs-strass}. For the solutions provided in \cite{Heule19}, after the first deflation step, we find that most solutions are singular. This provides a new numerical evidence that finding the tensor decomposition of low rank is hard \cite{deLim-08, HLim-13}. Some remarks and open problems are given in Section \ref{se-ropen}.

Our  numerical results are obtained by Matlab2022a. Some codes and data can be found on \cite{L-23}. The numerical experiments are taken on the server of ZJUT with the configuration: Intel(R) Xeon(R) Silver 4210R CPU running at 2.40GHz plus 2.39 GHz and 256 GB memory.

\section{Preliminaries}\label{sec-prelim}

Throughout this paper,  $\mathbb{C}$ is the complex field.
The vector space of $m\times n$ complex matrices
is denoted by $\mathbb{C}^{m\times n}$. Let $\mathbb{C}^n$ (or $\mathbb{C}^{n\times1}$) denote the $n$-dimensional column vector space over $\mathbb{C}$. Given a matrix $A\in \mathbb{C}^{m\times n}$. Let $\operatorname{rank}(A)$ denote the rank of $A$. The nullity of $A$ is
$$n-\operatorname{rank}(A),$$
which is the dimension of the null space of $A$ and denoted by $\mathscr{N}(A)$.

Let $F:\mathbb{C}^n\to \mathbb{C}^m$ be a polynomial map, so that there exist $m$ polynomials $f_i:\mathbb{C}^n\to \mathbb{C}$ ($i=1,2,\dots,m$) such that $F=[f_1,f_2,\dots,f_m]^\textbf{T}$. The solution set of $F$ over $\mathbb{C}$ is called \emph{an algebraic set} and denoted by $\textbf{V}(F)$, that is,
$$\textbf{V}(F)=\{\textbf{x}\in \mathbb{C}^n| f_1(\textbf{x})=f_2(\textbf{x})=\cdots=f_m(\textbf{x})=0\}.$$
Let $J(\textbf{x})$ denote the Jacobian matrix of $F(\textbf{x})$, so that $J(\textbf{x})\in \mathbb{C}^{m\times n}$. An algebraic set $\textbf{V}$  is reducible if there exist nonempty algebraic sets $V_1,V_2\subsetneqq\textbf{V}$ such that
$\textbf{V}=V_1\cup V_2$. $\textbf{V}$ is irreducible if no such decomposition exists. Each algebraic set $\textbf{V}$ can be decomposed into a unique finite union of irreducible
algebraic subsets:
\begin{equation*}
\textbf{V}=V_1\cup V_2\cup\cdots\cup V_r,
\end{equation*}
where $V_i$ is called the irreducible component of $\textbf{V}$.
For an irreducible component $V$, the set of manifold points in $V$ is connected and its dimension as a complex manifold is defined to be the dimension of $V$, denoted $\dim V$ \cite{BEHP-21,SW-05}. If
$\dim V=k$, $V$ is said to be a $k$-dimensional irreducible component.
The local dimension of a point $\textbf{x}\in \textbf{V}$ is the maximal dimension of the irreducible components of $\textbf{V}$ that contain $\textbf{x}$, which is denoted by $\dim_{\textbf{x}}\textbf{V}$.
For an irreducible component $V$, a property $P$ holds generically on $V$ if there is a nonempty Zariski open set $U\subseteq V$ such that $P$ holds on $V$. Each point of $U$ is called a generic point of $V$ with respect to $P$.

Considering $\textbf{V}(F)$ as a geometric object in $\mathbb{C}^n$, we have the following definition which makes a slight change of \cite[Def. 2.1]{HHS-23}.
\begin{definition}\label{def-geos}
A point $\textbf{x}_0\in \textbf{V}(F)$ is called \emph{geometric smooth}, if there is a unique irreducible component $V\subseteq  \textbf{V}(F)$ through $x_0$, and
$$ \dim T_{\textbf{x}_0}(V)=\dim V $$
where $T_{\textbf{x}_0}(V)$ is the tangent space of $V$ at $\textbf{x}_0$. Otherwise, $\textbf{x}_0$ is called a \emph{geometric singular} point of $\textbf{V}(F)$. The set of all geometric singular points
of $\textbf{V}(F)$ is denoted by $\operatorname{Sing}_G(\textbf{V}(F))$.
\end{definition}
It is well-known that geometric smooth points are generic on an irreducible component.
On the other hand,  we have the following definition.
\begin{definition}\label{def-nums}
Let $F(\textbf{x})$ be a polynomial system whose solution set and Jacobian are $\textbf{V}(F)$ and $J(\textbf{x})$, respectively.
A point $\textbf{x}_0\in \textbf{V}(F)$ is called \emph{numerical smooth}, if there is a unique irreducible component $V\subseteq  \textbf{V}(F)$ through $\textbf{x}_0$, and
$$ \mathscr{N}(J(\textbf{x}_0))=\dim V. $$

Otherwise, $\textbf{x}_0$ is called a \emph{numerical singular} point of $\textbf{V}(F)$. The set of all numerical singular points
of $\textbf{V}(F)$ is denoted by $\operatorname{Sing}_N(\textbf{V}(F))$.
\end{definition}
\begin{remark}
It is well-known that if $\textbf{x}_0\in\operatorname{Sing}_N(\textbf{V}(F))$, then
\begin{equation}
 \mathscr{N}(J(\textbf{x}_0))>\dim_{\textbf{x}_0} \textbf{V}(F).
\end{equation}

Comparing Definition \ref{def-geos} and \ref{def-nums}, we know that numerical smooth points are geometric smooth.
However, numerical singular points may contain geometric smooth points. For example, the irreducible components with multiplicity greater than one \cite[Example 8.10]{BHSW-13}.

For the polynomial system $F=[f_1,f_2,\dots,f_m]^\textsc{T}$, let $$\langle f_1,f_2,\dots,f_m \rangle\subseteq \mathbb{C}[x_1,x_2,\dots,x_n]$$
denote the ideal generated by $f_i$. On the other hand,
let $\textbf{I}(\textbf{V}(F))$ denote the vanishing ideal on $\textbf{V}(F)$:
$$\textbf{I}(\textbf{V}(F))=\{f\in \mathbb{C}[x_1,x_2,\dots,x_n]|
f(\textbf{x})=0~\text{for~all}~\textbf{x}\in \textbf{V}(F) \}.$$
By Proposition 2 of \cite[Sect.6 Chapter 9]{cox15} and  Corollary 9 of \cite[Sect.7 Chapter 9]{cox15}, we know that if
\begin{equation}
 \langle f_1,f_2,\dots,f_m \rangle=\textbf{I}(\textbf{V}(F)),
\end{equation}
then Definition \ref{def-geos} and \ref{def-nums} are equivalent, that is, the geometric smooth (resp. singular) points and  numerical smooth (resp. singular) points are the same. However, for any polynomial system $F=[f_1,f_2,\dots,f_m]^\textsc{T}$, it is well-known that
$$\langle f_1,f_2,\dots,f_m \rangle\subseteq\textbf{I}(\textbf{V}(F)),$$
and deciding if they are equal is usually hard. So, for those polynomial systems we are interested in, the geometric smooth points may be numerical singular.
\end{remark}

The definition of semiregularity and ultrasingularity of \cite{Zeng-23} are formalised to irreducible components as follows.
For the definition of intersection multiplicity of an irreducible component, we refer to \cite[Sec. 11.3]{BHSW-13}.

\begin{definition}
For a polynomial system $F$, given a $k$-dimensional  irreducible component $V\subseteq \textbf{V}(F)$.
If all geometric smooth points are numerical smooth, then we call $V$ a \emph{$k$-dimensional semiregular irreducible component}. That is, there is a nonempty Zariski open set $U\subseteq V$ such that
$$k=\dim_{\textbf{x}} \textbf{V}(F)=T_\textbf{x}V=\mathscr{N}(J(\textbf{x})),$$
for all $\textbf{x}\in U$. Each point in $U$  is called a $k$-dimensional semiregular  solution.

On the other hand, if all geometric smooth points are numerical singular and have the same intersection multiplicity, then we call $V$ a \emph{$k$-dimensional ultrasingular irreducible component}. That is, there is a nonempty Zariski open set $U\subseteq V$ and a constant positive number $n_0$ such that
$$k=\dim_{\textbf{x}} \textbf{V}(F)=T_\textbf{x} V<\mathscr{N}(J(\textbf{x}))=n_0,$$
for all $\textbf{x}\in U$.  Each point in $U$ is called a $k$-dimensional ultrasingular solution.
\end{definition}

\section{Deflation conjecture for the irreducible component}\label{sec:defl}


The deflation process raised in \cite{DZeng-05,DLZeng-11} focuses on isolated singular solutions.
In this section,  we generalize it to the case of $k$-dimensional ultrasingular irreducible components.
More details of the conjecture in \cite{Zeng-23} for ultrasingular solutions are given. Similar discussions are given in \cite{HW-13}.
In the following, the deflation process always refers to the deflation process raised in \cite{DZeng-05,DLZeng-11} or its generalization.

\subsection{The first step of the  deflation process}
Let $V\subseteq \textbf{V}(F)$ be  a $k$-dimensional ultrasingular irreducible component, which has a nonzero Zariski open set
\begin{equation}\label{eq-u0}
U_0\subseteq V \quad\text{with}\quad k<\mathscr{N}(J(\textbf{x}))=n_0
\end{equation}
for all $\textbf{x}\in U_0$. Fix a $k$-dimensional ultrasingular solution $\hat{\textbf{x}}_1\in U_0$. So we have
\begin{equation}\label{eq-n0}
k<\mathscr{N}\left(J(\hat{\textbf{x}}_1)\right)=n_0.
\end{equation}
Then, choose $R_1\in \mathbb{C}^{n_0\times n}$ and a vector $d_1\in \mathbb{C}^{n_0}$ randomly, such that the matrix
\begin{equation}\label{eq-JR1}
             \left[
             \begin{array}{c}
              J(\hat{\textbf{x}}_1)\\
              R_1\\
              \end{array}
              \right]
              \end{equation}
is of full (column) rank, and therefore
the linear system
\begin{equation}\label{eq-Jaclinsys}
    \left[\begin{array}{c}
        J(\hat{\textbf{x}}_1)\\
        R_1\\
        \end{array}\right]\textbf{x}=\left[\begin{array}{c}
                     \textbf{0}\\
                      d_1\\
                \end{array}
                        \right]
\end{equation}
has a unique solution $\textbf{x}_2=\hat{\textbf{x}}_2\neq \textbf{0}$. By the discussion in Section 3.1 of \cite{HW-13}, we can see that besides $\hat{\textbf{x}}_1$,  for almost all  $\hat{\textbf{x}}\in V$, the equation (\ref{eq-Jaclinsys}) also has a unique solution.
Let $\textbf{x}_2\in \mathbb{C}^n$ be new variables and fixing $R_1$ for all points in $V$. Then the first step of deflation of $F$ is defined by:
\begin{equation}\label{eq-def1}
F_1(\textbf{x}_1,\textbf{x}_2)=
\left[\begin{array}{c}
           F(\textbf{x}_1) \\
\left[\begin{array}{c}
J(\textbf{x}_1)\\R_1\\\end{array}\right]\textbf{x}_2
-\left[\begin{array}{c}\textbf{0}\\d_1\\\end{array}\right]                                                               \end{array}\right],~\text{for all}~\textbf{x}_1\in V.
\end{equation}
So ($\hat{\textbf{x}}_1,\hat{\textbf{x}}_2$) is a  solution of the new polynomial system $F_1(\textbf{x}_1,\textbf{x}_2)$.

When $m\geq n$, let $\pi_n:~\mathbb{C}^m\to\mathbb{C}^n$ denote the natural projection map, which is defined by
\begin{equation}\label{eq-pin}
\pi_n: (x_1,x_2,\dots,x_m)\mapsto (x_1,x_2,\dots,x_n).
\end{equation}
Let $\textbf{V}(F_1)$ denote the solution set of $F_1$.
For $U_0$ defined in (\ref{eq-u0}), let
\begin{equation}\label{eq-W1}
W_1=\{y\in \textbf{V}(F_1)|\pi_n(y)\in U_0\}\subset \textbf{V}(F_1),
\end{equation}
and $\overline{W_1}$ be the Zariski closure of $W_1$.
By the discussion of Lemma 4.6 of \cite{HW-13}, we also have the following lemma.
\begin{lemma}\label{lem-kn1n0}
$ \overline{W_1}$ is an irreducible component of $\textbf{V}(F_1)$ which is generically isomorphic to $V$. In particular,
$$\dim \overline{W_1}=\dim V=k.$$
\end{lemma}

 Let $J_1(\textbf{x}_1,\textbf{x}_{2})$ denote the Jacobian of $F_1(\textbf{x}_1,\textbf{x}_2)$. Then by Proposition 4 and Remark 5 of \cite{Lew-09}, with probability one, there is an open neighbourhood $\tilde{U}_1$ of
$(\hat{\textbf{x}}_1,\hat{\textbf{x}}_{2})$ such that the nullity
\begin{equation}\label{eq-n1}
\mathscr{N}(J_1(\textbf{x}_1,\textbf{x}_{2}))=\mathscr{N}(J_1(\hat{\textbf{x}}_1,\hat{\textbf{x}}_{2}))=n_1,
\end{equation}
for all $(\textbf{x}_1,\textbf{x}_2)\in \tilde{U}_1$. Then we let
\begin{equation}\label{eq-U1}
U_1=\overline{W}_1\cap \tilde{U}_1.
\end{equation}

\subsection{The second and subsequent steps}
In (\ref{eq-n1}), if $n_1=k$, then ($\hat{\textbf{x}}_1,\hat{\textbf{x}}_2$) is a $k$-dimensional semiregular zero of $F_1(\textbf{x}_1,\textbf{x}_2)$, so the singularity of $F(\textbf{x})$ at $\hat{\textbf{x}}_1$ and its neighbourhood is ``deflated''.
However, ($\hat{\textbf{x}}_1,\hat{\textbf{x}}_2$) may still be an ultrasingular zero of $F_1(\textbf{x}_1,\textbf{x}_2)$, that is, $n_1>k$.  In this case, we can repeat the same deflation process above for $F_1$ and obtain a new polynomial system $F_2$. More precisely, we can find and fix $R_2\in \mathbb{C}^{n_1\times 2n}$ and
$d_2\in \mathbb{C}^{n_1}$ such that
\begin{equation}\label{eq-def2}
F_2(\textbf{x}_1,\textbf{x}_2,\textbf{x}_3,\textbf{x}_4)=
\left[\begin{array}{c}
           F_1(\textbf{x}_1,\textbf{x}_2) \\
\left[\begin{array}{c}
J_1(\textbf{x}_1,\textbf{x}_2)\\R_2\\\end{array}\right]
\left[
  \begin{array}{c}
    \textbf{x}_3 \\
    \textbf{x}_4 \\
  \end{array}
\right]
-\left[\begin{array}{c}\textbf{0}\\d_2\\\end{array}\right]                                                               \end{array}\right]~\text{for all}~\textbf{x}_1\in V,
\end{equation}
and the solution ($\hat{\textbf{x}}_1,\hat{\textbf{x}}_2$) extends to a solution
($\hat{\textbf{x}}_1,\hat{\textbf{x}}_2,
\hat{\textbf{x}}_3,\hat{\textbf{x}}_4$) of $F_2$, uniquely.
Let $\textbf{V}(F_2)$ denote the solution set of $F_2$.
For $U_0$ and $\pi_n$ defined in (\ref{eq-u0}) and (\ref{eq-pin}) respectively, let
\begin{equation}\label{eq-W2}
W_2=\{y\in \textbf{V}(F_2)|\pi_n(y)\in U_0\}\subset \textbf{V}(F_2),
\end{equation}
and $\overline{W_2}$ be the Zariski closure of $W_2$. Then following the same reason as Lemma \ref{lem-kn1n0}, we also have $\dim \overline{W_2}=k$.
Just as (\ref{eq-U1}), there is also a neighbourhood $U_2\subseteq \overline{W_2}$ of $(\hat{\textbf{x}}_1,\hat{\textbf{x}}_{2},
\hat{\textbf{x}}_3,\hat{\textbf{x}}_{4})$  such that
the nullity
\begin{equation}\label{eq-n2}
\mathscr{N}(J_2(\textbf{x}_1,\textbf{x}_{2},
\textbf{x}_3,\textbf{x}_4))=
\mathscr{N}(J_2(\hat{\textbf{x}}_1,\hat{\textbf{x}}_{2},
\hat{\textbf{x}}_3,\hat{\textbf{x}}_{4}))=n_2,
\end{equation}
for all $(\textbf{x}_1,\textbf{x}_2,\textbf{x}_3,\textbf{x}_4)\in U_2$.

Generally, assume
($\hat{\textbf{x}}_1,\dots,\hat{\textbf{x}}_{2^{i}}$)
is an ultrasingular zero of $F_i(\textbf{x}_1,\dots,\textbf{x}_{2^{i}})$
after $i$ steps of deflation.
Let
\begin{equation}\label{eq-Wi}
W_i=\{y\in \textbf{V}(F_i)|\pi_n(y)\in U_0\}\subset \textbf{V}(F_i),
\end{equation}
and $\overline{W_i}$ be the Zariski closure of $W_i$. Then
$\dim\overline{W_i}=k$
and the Jacobian
$J_i(\textbf{x}_1,\dots,\textbf{x}_{2^{i}})$ maintains a nullity $n_i>k$, in a neighborhood $U_i\subset \overline{W_i}$ of $(\hat{\textbf{x}}_1,\dots,\hat{\textbf{x}}_{2^{i}})$. That is,
\begin{equation}\label{eq-ni}
\mathscr{N}\left(J_i(\textbf{x}_1,\dots,\textbf{x}_{2^{i}})
\right)=\mathscr{N}\left(J_i(\hat{\textbf{x}}_1,\dots,\hat{\textbf{x}}_{2^{i}})
\right)=n_i
\end{equation}
for all $(\textbf{x}_1,\dots,\textbf{x}_{2^{i}})\in U_i$.
Then the next deflation step expands the system to
\begin{equation}\label{eq-def-i}
F_{i+1}(\textbf{x}_1,\textbf{x}_2,\dots,\textbf{x}_{2^{i+1}})=
\left[\begin{array}{c}
      F_{i}(\textbf{x}_1,\dots,\textbf{x}_{2^{i}})\\
\left[\begin{array}{c}
J_{i}(\textbf{x}_1,\dots,\textbf{x}_{2^{i}})
\\R_{i+1}\\\end{array}\right]
\left[\begin{array}{c}
\textbf{x}_{2^{i}+1} \\
 \vdots \\
 \textbf{x}_{2^{i+1}}\\
 \end{array}\right]
-\left[\begin{array}{c}\textbf{0}\\d_{i+1}\\\end{array}\right]                                                               \end{array}\right],
\end{equation}
for all $\textbf{x}_1\in V$, where $R_{i+1}\in \mathbb{C}^{n_i\times 2^in}$ and
$d_{i+1}\in\mathbb{C}^{n_{i}}$ are (randomly) chosen and fixed such that almost all points of $\overline{W_i}$ extend to a solution of $F_{i+1}$, uniquely.

\begin{theorem}\label{thm-nulldi}
In the deflation process above, for $k$, $n_0$ and $n_i$, we have
\begin{equation}\label{eq-dni}
k\leq n_i\leq\cdots\leq n_1\leq n_0,
\end{equation}
for each $i\geq1$. Moreover, suppose that $J_{i}(\hat{\textbf{x}}_1,\dots,\hat{\textbf{x}}_{2^{i}}) \in \mathbb{C}^{a_i\times 2^{i} n}$ for some $a_i$. Let $J_{i}(\hat{\textbf{x}}_1)\in \mathbb{C}^{a_i\times n}$ be the submatrix of $J_{i}(\hat{\textbf{x}}_1,\dots,\hat{\textbf{x}}_{2^{i}})$ which is obtained by taking the first $n$ columns of $J_{i}(\hat{\textbf{x}}_1,\dots,\hat{\textbf{x}}_{2^{i}})$.
Then
\begin{equation}\label{eq-nin-J}
n_i=n-rank~J_i(\hat{\textbf{x}}_1).
\end{equation}
\end{theorem}
\begin{proof}
By Lemma \ref{lem-kn1n0}, we can see that without changing its dimension, after each step of deflation, the $k$-dimensional ultrasingular irreducible component is embedded into a new $k$-dimensional irreducible component. So we have $k\leq n_i$ for each $i$.

For $n_i\leq\cdots\leq n_1\leq n_0$, it suffices to show that $n_1\leq n_0$ by the recursive process of the deflation.
From (\ref{eq-def1}), we know that
\begin{equation}\label{eq-J1hat}
J_1(\hat{\textbf{x}}_1,\hat{\textbf{x}}_{2})=
\left[\begin{array}{cc}
 J(\hat{\textbf{x}}_1) & \textbf{0} \\
J\left((J(\hat{\textbf{x}}_1)\hat{\textbf{x}}_2), \hat{\textbf{x}}_1\right)& J(\hat{\textbf{x}}_1)\\
                 \textbf{0}    & R_1\\
                           \end{array}
                 \right]\in \mathbb{C}^{(2m+n_0)\times 2n},
\end{equation}
where $J\left((J(\hat{\textbf{x}}_1)\hat{\textbf{x}}_2), \hat{\textbf{x}}_1\right)$ is the partial Jacobian of $J(\textbf{x}_1)\textbf{x}_2$  with respect to $\textbf{x}_1$ which
takes the value at $(\hat{\textbf{x}}_1,\hat{\textbf{x}}_2)$. By the choice of $R_1$, we have $$rank~J_1(\hat{\textbf{x}}_1,\hat{\textbf{x}}_{2})\geq n_0+n.$$ So  $n_1=2n-rank~J_1(\hat{\textbf{x}}_1,\hat{\textbf{x}}_{2})\leq n_0.$

For (\ref{eq-nin-J}), when $i=0$, the proof is followed by  $J_0(\hat{\textbf{x}}_1)=J(\hat{\textbf{x}}_1)$.
For $i\geq1$, from (\ref{eq-def-i}) and the choice of $R_{i}$ in each step, we only show the case when $i=1$. The others are similar too.

When $i=1$, for $J_1(\hat{\textbf{x}}_1,\hat{\textbf{x}}_2)$, from (\ref{eq-J1hat}), we have
\begin{equation}\label{eq-Jx1}
J_1(\hat{\textbf{x}}_1)=
\left[\begin{array}{c}
J(\hat{\textbf{x}}_1)\\
J\left((J(\hat{\textbf{x}}_1)\hat{\textbf{x}}_2),\hat{\textbf{x}}_1\right)\\
\textbf{0} \\
\end{array}\right].
\end{equation}

On the other hand, by choice of $R_1$,  we can see that the submatrix  obtained by taking columns $n+1$ to $2n$:
$$ \left[
     \begin{array}{c}
       \textbf{0} \\
        J(\hat{\textbf{x}}_1) \\
        R_1 \\
     \end{array}
   \right]
$$
is of full column rank  $n$. So $$n_1=2n-rank~J_1(\hat{\textbf{x}}_1,\hat{\textbf{x}}_2)
=n-rank~J_1(\hat{\textbf{x}}_1).$$
\end{proof}
\begin{remark}\label{rem-sinlowb}
For the  points in $V\setminus U_0$ which are geometric singular, we can also apply the deflation process defined in (\ref{eq-def-i}) to them.
Similar to (\ref{eq-dni}),
with the same reason, we still have
$$n_i\leq\cdots\leq n_1\leq n_0.$$
However, for  points in $V\setminus U_0$, we can see that
$$n_i<k$$
can happen for some $i$, see Example \ref{exm-cusp} and \ref{exm-whitney}. So the deflation process provides a lower bound of local dimension for them.

From (\ref{eq-Jx1}), we can see that after each deflation step, the addition of new nonzero submatrices to the bottom of the first $n$ columns is the reason that makes the nullities decrease.
\end{remark}

\subsection{The deflation conjecture}
By Theorem \ref{thm-nulldi}, we can see that the nullity $n_i$ can not be always strictly decreased. It should be stabilized after a finite number of steps.
\begin{definition}\label{def-defseq}
Suppose that $\hat{\textbf{x}}$ is a solution of $F(\textbf{x})=0$.
With notations above, for the deflation process (\ref{eq-def-i}), let $i_0\geq0$ be the smallest integer such that
$$ n_{i+1}=n_{i} \qquad\text{for~all}\quad i\geq i_0.$$
Then
\begin{equation}\label{eq-defsequence}
(n_0,n_1,\dots,n_{i_0})
\end{equation}
is called the \emph{deflation sequence} of $\hat{\textbf{x}}$.
\end{definition}

\begin{definition}\label{def-regular}
Let $\hat{\textbf{x}}\in \textbf{V}(F)$ with local dimension
$\dim_{\hat{\textbf{x}}} \textbf{V}(F)=k$. Let  $(n_0,n_1,\dots,n_{i_0})$ denote its deflation sequence.
If
$$n_{i_0}=k,$$
then we say that the deflation sequence $(n_0,n_1,\dots,n_{i_0})$ is \emph{regular} and $\hat{\textbf{x}}$ has \emph{a regular deflation process}.
\end{definition}

\begin{example}\label{exm-cusp}
Let $F(x,y)=y^2-x^3$ with  $z_1=(0,0)$ and $z_2=(1,1)$ in $\textbf{V}(F)$.
By direct computer verification,  the first 4 numbers
of their deflation sequences are
$$(n_0,n_1,n_2,n_3)=(2,1,1,0)~\text{and}~(1,1,1,1),~\text{respectively}.$$
\end{example}

\begin{example}\label{exm-whitney}
Let $F(x,y,z)=x^2-y^2z$ be the Whitney umbrella. Consider the three points of $\textbf{V}(F)$: (2, 2, 1), (0, 0, 1) and (0, 0, 0). The first 4 numbers of their deflation sequences are:
$$(n_0,n_1,n_2,n_3)=(2,2,2,2),~(3,2,1,1)~\text{and}~(3,2,2,1),
~\text{respectively}.$$
\end{example}

It is proved in \cite{DZeng-05,LVZ-06} that all isolated singular solutions have regular deflation processes. For $k$-dimensional ultrasingular irreducible components, Zeng's conjecture in \cite[Sect. 9]{Zeng-23} can be formalised as follows.

\begin{conjecture}[\emph{Deflation Conjecture}]\label{conj-dc}
Let $V\subseteq \textbf{V}(F)$ be a $k$-dimensional
ultrasingular irreducible component. So there is a nonempty Zariski open set $U\subseteq V$ which consists of geometric smooth points and a constant number $n_0$ such that $\mathscr{N}\left(J(\textbf{x})\right)=n_0>k$ for all $\textbf{x}\in U$.  We have:
\begin{enumerate}
  \item The recursive deflation process (\ref{eq-def-i}) terminates in finitely many steps such that the deflation sequences of all $\textbf{x}\in U$ are equal and regular.
  \item
For the point $\hat{\textbf{x}}\in V\setminus U$ with deflation sequence $(n_0,n_1,\dots,n_{i_0})$, $n_{i_0}$ is a lower bound of its local dimension. That is,
$$\dim_{\hat{\textbf{x}}}\textbf{V}(F)\geq n_{i_0}.$$
\end{enumerate}
\end{conjecture}

\begin{remark}
In \cite{HW-13}, using the results in \cite{Bates10}, Hauenstein and Wampler showed that Conjecture \ref{conj-dc} holds for their strong deflation. Leykin et al.'s deflation process can also be given  a similar conjecture \cite{LVZ-06}, and similarly for other deflation processes in the literatures.
\end{remark}

\begin{remark}
If Conjecture \ref{conj-dc} is true,  the number in the deflation sequence provides upper bounds of the local dimensions of geometric smooth points. The last number of the deflation sequence is a lower bound of the local dimensions of the geometric singular points. This leads to the problem of determining the singularity of a solution (see also Section \ref{subs-desin}).
\end{remark}

\begin{remark}
Given a set of solutions of a polynomial system, we usually don't know the relations between them, such as if they lie on a same irreducible component. So if we apply the deflation process to them, the $R_i$ and $d_i$ defined in (\ref{eq-def-i}) are chosen separately for each solution.   Thus, in practical numerical experiments, if we don't know the relation between the solutions, the deflation process is the same as the deflation process for isolated solutions in \cite{DZeng-05,DLZeng-11}. For example, the solutions of Brent equations studied in Section \ref{se-numerexp}.
\end{remark}

\section{On realization of the deflation process}\label{se-realization}

In this section, the details of the construction of the Jacobians of the first 3 steps of the deflation process raised in \cite{DZeng-05, DLZeng-11} are given. Although they may not be new, we have not seen them in the literatures. We can see that the Jacobians are not hard to be realized numerically. So it is a useful tool for the determination of local dimensions of the algebraic sets.

Firstly, we introduce some notations. The original variable
$\textbf{x}=[x_1,\cdots,x_n]^\textbf{T}$ will be denoted by $\textbf{x}_1$ in accordance with the notation for the auxiliary
(vector) variables $\textbf{x}_2$, $\textbf{x}_3$,... etc. For any fixed or variable vector $\textbf{y}=[y_1,\dots,y_n]^\textbf{T}$,
the directional differentiation operator along $\textbf{y}$ is defined as
\begin{equation*}
\nabla_\textbf{y}=y_1\frac{\partial}{\partial x_1}+y_2\frac{\partial}{\partial x_2}+\cdots +y_n\frac{\partial}{\partial x_n}.
\end{equation*}
\begin{lemma}\label{lem-commuta}
The composition of $\nabla_\textbf{y}$ and $\nabla_\textbf{z}$ is commutative, that is,
\begin{equation*}
\nabla_\textbf{y}  \nabla_\textbf{z}=\nabla_\textbf{z}  \nabla_\textbf{y}
\end{equation*}
\end{lemma}
\begin{proof}
Since $\frac{\partial^2}{\partial x_i \partial x_j}=\frac{\partial^2}{\partial x_j \partial x_i}$, we have
\begin{equation*}
\begin{split}
\nabla_\textbf{y} \nabla_\textbf{z}&=\nabla_\textbf{y}\left(
z_1\frac{\partial}{\partial x_1}+z_2\frac{\partial}{\partial x_2}+\cdots +z_n\frac{\partial}{\partial x_n}\right)\\
&=\sum_{i=1}^n\sum_{j=1}^n y_i z_j \frac{\partial^2}{\partial x_i \partial x_j}=\sum_{i=1}^n\sum_{j=1}^n z_i y_j \frac{\partial^2}{\partial x_i \partial x_j}\\
&=\nabla_\textbf{z} \nabla_\textbf{y}.
\end{split}
\end{equation*}
\end{proof}

For any variable $\textbf{u}=(u_1,\dots,u_n)^\textbf{T}$, the gradient operator is defined by
$$ \triangle_{\textbf{u}}=\left[\frac{\partial}{\partial u_1},\frac{\partial}{\partial u_2},\cdots,\frac{\partial}{\partial u_n}\right],$$
whose ``dot product'' with a vector $\textbf{v}=[v_1,\dots,v_n]^\textbf{T}$ is defined as
\begin{equation*}
\textbf{v}\cdot\triangle_{\textbf{u}}=v_1\frac{\partial}{\partial u_1}+v_2\frac{\partial}{\partial u_2}+\cdots +v_n\frac{\partial}{\partial u_n}.
\end{equation*}
So we have $\nabla_\textbf{y}=\textbf{y}\cdot\triangle_{\textbf{x}}
=\textbf{y}\cdot\triangle_{\textbf{x}_1}$.

Let  $J(\textbf{x})$ be the Jacobian of the polynomial system $F(\textbf{x})=[f_1(\textbf{x}),f_2(\textbf{x}),\dots,
f_m(\textbf{x})]^\textbf{T}$.
Then $J(\textbf{x})\textbf{y}$ can be denoted by
\begin{equation*}
J(\textbf{x})\textbf{y}=\nabla_{\textbf{y}}F
=\left[
   \begin{array}{c}
     \nabla_{\textbf{y}}f_1 \\
     \nabla_{\textbf{y}}f_2 \\
     \vdots \\
     \nabla_{\textbf{y}}f_m\\
   \end{array}
 \right].
\end{equation*}
Similarly, set
$$\nabla_{\textbf{u}}\nabla_{\textbf{v}}F=
\left[
   \begin{array}{c}
     \nabla_{\textbf{u}}\nabla_{\textbf{v}}f_1 \\
     \nabla_{\textbf{u}}\nabla_{\textbf{v}}f_2 \\
     \vdots \\
     \nabla_{\textbf{u}}\nabla_{\textbf{v}}f_m\\
   \end{array}
 \right]
.$$

Let $G(\textbf{x}_1,\textbf{x}_2,\dots,\textbf{x}_k):\mathbb{C}^{kn}\to \mathbb{C}^m$ be a polynomial system and $G=[g_1,g_2,\dots,g_m]^\textbf{T}$, its \emph{partial Jacobian} with respect to some $\textbf{x}_i=[x_{i1},\dots,x_{in}]^{\textbf{T}}$ is defined by
\begin{equation}\label{eq-defPJ}
J(G,\textbf{x}_i):=\left[
   \begin{array}{c}
     \triangle_{\textbf{x}_i}g_1 \\
     \triangle_{\textbf{x}_i}g_2 \\
     \vdots \\
     \triangle_{\textbf{x}_i}g_m\\
   \end{array}
 \right]=\left[
           \begin{array}{ccc}
\frac{\partial g_1}{\partial x_{i1}} & \cdots &\frac{\partial g_1}{\partial x_{in}}  \\
             \vdots &  & \vdots \\
\frac{\partial g_m}{\partial x_{i1}} & \cdots &\frac{\partial g_m}{\partial x_{in}} \\
           \end{array}
         \right].
\end{equation}

Let $J(J(\textbf{x})\textbf{v},\textbf{x})$ be the partial Jacobian of $J(\textbf{x})\textbf{v}$ with respect to $\textbf{x}$, and similarly  for $J(J(J(\textbf{x})\textbf{w},\textbf{x})\textbf{v},\textbf{x})$.
Then we have the following lemma.
\begin{lemma}\label{lem-comnab}
\begin{equation*}
J(J(\textbf{x})\textbf{v},\textbf{x})\textbf{u}
=\nabla_{\textbf{u}}\nabla_{\textbf{v}}F
=\left[
   \begin{array}{c}
     \nabla_{\textbf{u}}\nabla_{\textbf{v}}f_1 \\
     \nabla_{\textbf{u}}\nabla_{\textbf{v}}f_2 \\
     \vdots \\
    \nabla_{\textbf{u}}\nabla_{\textbf{v}}f_m\\
   \end{array}
 \right]=\nabla_{\textbf{v}}\nabla_{\textbf{u}}F.
\end{equation*}

\begin{equation*}
J(J(J(\textbf{x})\textbf{w},\textbf{x})\textbf{v},\textbf{x})\textbf{u}
=\nabla_{\textbf{u}}\nabla_{\textbf{v}}\nabla_{\textbf{w}}F
=\left[
   \begin{array}{c}
     \nabla_{\textbf{u}}\nabla_{\textbf{v}}\nabla_{\textbf{w}}f_1 \\
     \nabla_{\textbf{u}}\nabla_{\textbf{v}}\nabla_{\textbf{w}}f_2 \\
     \vdots \\
    \nabla_{\textbf{u}}\nabla_{\textbf{v}}\nabla_{\textbf{w}}f_m\\
   \end{array}
 \right]=\nabla_{\textbf{a}}\nabla_{\textbf{b}}\nabla_{\textbf{c}}F,
\end{equation*}
where $\textbf{a}$, $\textbf{b}$, $\textbf{c}$ is any permutation of $\textbf{u}$, $\textbf{v}$, $\textbf{w}$.
\end{lemma}
\begin{proof}
The proof is followed by Lemma \ref{lem-commuta}. Since  for $i=1,2,\dots,m$, we have $$\nabla_{\textbf{u}}\nabla_{\textbf{v}}f_i=
\nabla_{\textbf{v}}\nabla_{\textbf{u}}f_i$$
and
$$\nabla_{\textbf{u}}\nabla_{\textbf{v}}\nabla_{\textbf{w}}f_i
=\nabla_{\textbf{a}}\nabla_{\textbf{b}}\nabla_{\textbf{c}}f_i,$$
where $\textbf{a}$, $\textbf{b}$, $\textbf{c}$ is any permutation of $\textbf{u}$, $\textbf{v}$, $\textbf{w}$.
\end{proof}

For the  partial Jacobians of  $\nabla_{\textbf{u}}F(\textbf{x})$,
$\nabla_{\textbf{u}}\nabla_{\textbf{v}}F(\textbf{x})$ and $\nabla_{\textbf{u}}\nabla_{\textbf{v}} \nabla_{\textbf{w}} F(\textbf{x})$, we have the following formulas.
\begin{lemma}\label{lem-Jacom}
With notation in (\ref{eq-defPJ}), we have
\
\begin{enumerate}
\item $J(\nabla_{\textbf{u}}F(\textbf{x}),\textbf{v})=
\textbf{0}$,

\item $J(\nabla_{\textbf{u}}F(\textbf{x}),\textbf{u})=
J(\textbf{x})$,

\item
$J(\nabla_{\textbf{v}}\nabla_{\textbf{u}}F(\textbf{x}),\textbf{u})=
J(\nabla_{\textbf{u}}\nabla_{\textbf{v}}F(\textbf{x}),\textbf{u})=
J(\nabla_{\textbf{v}}F(\textbf{x}),\textbf{x})$,

\item
$J(\nabla_{\textbf{v}}\nabla_{\textbf{w}}
\nabla_{\textbf{u}}F(\textbf{x}),
\textbf{u})=
J(\nabla_{\textbf{v}}\nabla_{\textbf{u}}\nabla_{\textbf{w}}
F(\textbf{x}),\textbf{u})=
J(\nabla_{\textbf{u}}\nabla_{\textbf{v}}\nabla_{\textbf{w}}
F(\textbf{x}),\textbf{u})=
J(\nabla_{\textbf{v}}\nabla_{\textbf{w}}
F(\textbf{x}),\textbf{x})$.
\end{enumerate}
\end{lemma}
\begin{proof}
The proof is followed by direct verification and Lemma \ref{lem-comnab}.
\end{proof}

With the preparation above, in the following we describe the first 3 steps of the deflation process by an algorithm. From this, we can obtain the first 4 numbers ($n_0,n_1,n_2,n_3$) of the deflation sequence.

\begin{algorithm}
\caption{Generate the first 4 numbers of the deflation sequence}
\begin{algorithmic}[1]
\REQUIRE A polynomial system $F(\textbf{x})=[f_1(\textbf{x}),f_2(\textbf{x}),\dots,
f_m(\textbf{x})]^\textbf{T}$ and a solution $\hat{\textbf{x}}\in \textbf{V}(F)\subseteq \mathbb{C}^n$
\ENSURE The first 4 numbers of the deflation sequence: ($n_0$, $n_1$, $n_2$, $n_3$)
\STATE Set $J_0(\textbf{x}_1)=J(\textbf{x})$, $F_0(\textbf{x}_1)=F(\textbf{x})$ and $\hat{\textbf{x}}_1=\hat{\textbf{x}}$

\STATE Set $n_0=n-\operatorname{rank}(J_0(\hat{\textbf{x}}_1))$

\FOR{$i$=1, 2, 3}

\STATE Choose $R_{i}\in \mathbb{C}^{n_{i-1}\times 2^{i-1}n}$ and
$d_{i}\in\mathbb{C}^{n_{i-1}}$  randomly

\STATE Expand the system $F_{i-1} (\textbf{x}_1,\dots,\textbf{x}_{2^{i-1}})$ to $F_{i}(\textbf{x}_1,\dots,\textbf{x}_{2^{i}})$ by equation (\ref{eq-def-i})

\STATE Calculate the Jacobian matrix of $F_{i}(\textbf{x}_1,\dots,\textbf{x}_{2^{i}})$ by Lemma \ref{lem-comnab} and \ref{lem-Jacom}; denote it by $J_{i}(\textbf{x}_1,\dots,\textbf{x}_{2^{i}})$

\STATE Fixing $\hat{\textbf{x}}_1,\dots,\hat{\textbf{x}}_{2^{i-1}}$,
solve the  (linear) system
$F_{i}(\hat{\textbf{x}}_1,\dots, \hat{\textbf{x}}_{2^{i-1}},
\textbf{x}_{2^{i-1}+1},\dots,\textbf{x}_{2^{i}})=0$;
let $(\hat{\textbf{x}}_{2^{i-1}+1},\dots,\hat{\textbf{x}}_{2^{i}})$ be the (unique) solution

\STATE Calculate $\operatorname{rank}(J_i(\hat{\textbf{x}}_1,\hat{\textbf{x}}_2,
\dots,\hat{\textbf{x}}_{2^{i}}))$

\STATE Set $n_i=2^in-\operatorname{rank}(J_i(\hat{\textbf{x}}_1,\hat{\textbf{x}}_2,
\dots,\hat{\textbf{x}}_{2^{i}}))$

\ENDFOR
\end{algorithmic}
\end{algorithm}

For the polynomial system  $F=[f_1, f_2,\dots,f_m]^{\textsc{T}}$, the degree of $F$ is defined by $$\textrm{deg}~F=\max\{\textrm{deg}~f_i|i=1,2,\dots,m\},$$
where $\textrm{deg}~f_i$ are the degrees of the polynomials $f_i$.
So from above analysis, we have the following proposition.
\begin{proposition}
If $\textrm{deg}~F=3$, in the deflation process (\ref{eq-def-i}), the Jacobians $J_{i}(\textbf{x}_1,\dots,\textbf{x}_{2^{i}})$ are builded up by $R_i$ and the following partial Jacobians
$$J(\textbf{x}),~J(\nabla_{\textbf{u}} F(\textbf{x}),\textbf{x}),~J(\nabla_{\textbf{u}} \nabla_{\textbf{v}} F(\textbf{x}),\textbf{x}),$$
where $\textbf{u},~\textbf{v}$ are (partial) solutions of the extended system.
\end{proposition}
\begin{proof}
Since $\textrm{deg}~F=3$, we can see that $J(\nabla_{\textbf{u}} \nabla_{\textbf{v}}\nabla_{\textbf{w}} F(\textbf{x}),\textbf{x})=\textbf{0}$, for any solutions $\textbf{u}$, $\textbf{v}$,  $\textbf{w}$ of the extended system.
\end{proof}

The degree of the polynomial systems obtained by tensor decompositions of tensors of order 3 are all of degree 3, such as the Brent equations in next section.

\section{Numerical experiments on Brent equations}\label{se-numerexp}

Firstly, we recall the definition of Brent equations (see also \cite{LBZ-23}).
Let $V_1=\mathbb{C}^{m\times n}$, $V_2=\mathbb{C}^{n\times p}$ and $V_3=\mathbb{C}^{p\times m}$ denote the set of complex $m\times n$, $n\times p$  and $p\times m$ matrices, respectively.
The matrix multiplication tensor is defined by (see e.g. \cite{Bur15,Lands-17})
\begin{equation}\label{eq-mmt}
\langle m,n,p \rangle=\sum_{i=1}^m\sum_{j=1}^n\sum_{k=1}^p e_{ij}\otimes e_{jk}\otimes e_{ki}\in \mathbb{C}^{m\times n}\otimes \mathbb{C}^{n\times p}\otimes \mathbb{C}^{p\times m},
\end{equation}
where $e_{ij}$ denotes the matrix with a 1 in the $i$th row and $j$th column, and other entries are 0. Finding a tensor decomposition of $\langle m,n,p \rangle$ with rank $r$ is equivalent to find solutions of a system of polynomial equations called the \emph{Brent equations}, which is given as follows:
\begin{equation}\label{eq-brent}
\sum_{i=1}^{r} \alpha_{i_1,i_2}^{(i)}\beta_{j_1,j_2}^{(i)}\gamma_{k_1,k_2}^{(i)}=
\delta_{i_2,j_1}\delta_{j_2,k_1}\delta_{k_2,i_1},
\end{equation}
where $k_2,i_1\in \{1,2,\dots,m\}$, $i_2,j_1\in\{1,2,\dots,n\}$,
$j_2,k_1\in\{1,2,\dots,p\}$, $\delta_{ij}$ is the Kronecker symbol, $\alpha_{i_1,i_2}^{(i)}$, $\beta_{j_1,j_2}^{(i)}$ and $\gamma_{k_1,k_2}^{(i)}$ are unknown variables.

In the following, the polynomial system in (\ref{eq-brent}) is called \emph{the Brent equations of type $[m,n,p|r]$}, which is denoted by  $B(m,n,p|r)$. We can see that the polynomial system $B(m,n,p|r)$ has $(mn+np+pm)r$ variables and $(mnp)^2$ equations.
So $B(m,n,p|r)$ can be considered as a
polynomial mapping:
$$B(m,n,p|r)\colon \mathbb{C}^{(mn+np+pm)r}\to \mathbb{C}^{(mnp)^2}.$$
Usually, $B(m,n,p|r)$ is an overdetermined polynomial system. The solution set of $B(m,n,p|r)$ over $\mathbb{C}$ is denoted by $\textbf{V}(m,n,p|r)$.


In the following of this section, we report our experimental results.  By Theorem \ref{thm-nulldi} and Remark \ref{rem-sinlowb}, the deflation sequence tells us the bound of the local dimension. We can see that the as $r$ in (\ref{eq-brent}) decreases, singular solutions become more and more.

\subsection{Deflation test on $\textbf{V}(2,2,2|7)$: Strassen's classical solution}\label{secs-strass}
\


For $B(2,2,2|7)$, it is an underdetermined system which has $m=64$ equations and $n=84$ variables.
So from equation (3) of \cite{WHS-11}, we know that for any
$p\in\textbf{V}(2,2,2|7)$, its local dimension has a lower bound
$$20=84-64=n-m\leq\dim_p\textbf{V}(2,2,2|7).$$
In particular, for the classical solution of $B(2,2,2|7)$ given by Strassen \cite{Stra-69} (see also \cite{Bur15,Lands-17}), after doing deflation 3 times, the first 4 numbers of  deflation sequence are
\begin{equation}\label{eq-ds222}
(23, 23, 23, 23).
\end{equation}
In fact, its local dimension is 23. In the following, we will give a detailed explanation.

\begin{remark}\label{rem-2227}\footnote{We are grateful to the anonymous referee for improving this remark.}
Suppose that
\begin{equation}\label{eq-mnp}
\langle m,n,p \rangle=u_1\otimes v_1\otimes w_1+u_2\otimes v_2\otimes w_2+\cdots+u_r\otimes v_r\otimes w_r,
\end{equation}
where $u_i\in\mathbb{C}^{m\times n}$, $v_i\in\mathbb{C}^{n\times p}$, $w_i\in\mathbb{C}^{p\times m}$. We have the following group actions on   $\textbf{V}(m, n, p |r)$.
\begin{enumerate}
  \item
Let $a \in GL_m(\mathbb{C}), b \in GL_n(\mathbb{C})$, and $c\in
GL_p(\mathbb{C})$. It is well-known that the transformation $T(a, b, c)$ of $\mathbb{C}^{m\times n} \otimes \mathbb{C}^{n\times p} \otimes \mathbb{C}^{p\times n}$, defined by
$$x \otimes y \otimes z \mapsto a x b^{-1} \otimes b y c^{-1} \otimes c z a^{-1}$$
preserves the tensor $\langle m, n, p\rangle$ (see e. g. \cite{Bur15,deGr78-1,deGr78-2}). So the transformation
$$x \mapsto\left(\left(a u_1 b^{-1}, b v_1 c^{-1}, c w_1 a^{-1}\right), \ldots,\left(a u_r b^{-1}, b v_r c^{-1}, c w_r a^{-1}\right)\right)$$
preserves $\textbf{V}(m, n, p |r)$. Therefore the group $GL_m(\mathbb{C}) \times GL_n(\mathbb{C}) \times GL_p(\mathbb{C})$ acts on $\textbf{V}(m, n, p |r)$. The kernel (stabilizer) of this action consists of all triples $\left(\lambda I_m, \lambda I_n, \lambda I_p\right)$, where $\lambda \in \mathbb{C}^*=\mathbb{C}\setminus \{0\}$. Denote this group of transformations by $G_1$. Clearly,
\begin{align*}
 \dim G_1=&\dim GL_m(\mathbb{C})+\dim GL_n(\mathbb{C})+\dim GL_p (\mathbb{C})-1 \\
 =&m^2+n^2+p^2-1.
 \end{align*}
From Theorem 4.12 of \cite{Bur15}, besides $G_1$, there is another discrete group  $Q(m,n,p)$ that acts on $\textbf{V}(m, n, p |r)$. $Q(m,n,p)$ is a subgroup of $S_3$ depending on $m$, $n$ and $p$. It is well-known that the isotropy group of $\langle m,n,p \rangle$ consists of
$G_1$ and $Q(m,n,p)$.

\item
Let
     \begin{equation}\label{eq-lmc}
     \lambda=(\lambda_1,\lambda_2,\ldots,\lambda_r),~\mu=
     (\mu_1,\mu_2,\ldots,\mu_r)
     \in(\mathbb{C}^*)^r.
     \end{equation}
 Then we have
\begin{align}
\langle m,n,p \rangle = & (\lambda_1u_1)\otimes (\mu_1v_1)\otimes\left((\lambda_1\mu_1)^{-1}w_1\right)
  +(\lambda_2u_2)\otimes (\mu_2v_2)\otimes \left((\lambda_2\mu_2)^{-1}w_2\right)\notag\\
   &+\cdots+(\lambda_ru_r)\otimes(\mu_rv_r)\otimes \left((\lambda_r\mu_r)^{-1}w_r\right)\label{eq-mnplm}.
\end{align}
It is not hard to see that $(\mathbb{C}^*)^r$ can be considered as the set of nonsingular diagonal matrices of size $r\times r$ which is a linear algebraic group of dimension $r$. Let $G_2=(\mathbb{C}^*)^{2r}=(\mathbb{C}^*)^{r}\times (\mathbb{C}^*)^{r}$. Then
by (\ref{eq-lmc}) and (\ref{eq-mnplm}), we can see that there is an action of the linear algebraic group $G_2$ on $\textbf{V}(m,n,p;r)$. The dimension of $G_2$ is
           $$\dim G_2=2r.$$

\item From (\ref{eq-mnp}) we have
\begin{equation}\label{eq-sr}
  \langle m,n,p \rangle=u_{i_1}\otimes v_{i_1}\otimes w_{i_1}+u_{i_2}\otimes v_{i_2}\otimes w_{i_2}+\cdots+u_{i_r}\otimes v_{i_r}\otimes w_{i_r},
\end{equation}
where ($i_1$, $i_2$,...,$i_r$) is a permutation of (1,2,...,$r$). Let $S_r$ be the symmetric group of order $r$ and denote  $G_3=S_r$. So from (\ref{eq-sr}), there is an action of the linear algebraic group $G_3$ on $\textbf{V}(m,n,p|r)$. The dimension of $G_3$ is $$\dim G_3=0.$$
\end{enumerate}

We can see that for all tensors $t$ of order 3, the actions of $G_2$ and $G_3$ always exist. Let $G$ be the group of transformations of $\textbf{V}(m,n,p|r)$ generated by $G_1$, $Q(m,n,p)$, $G_2$ and $G_3$. It is easy to see that $G_1$ and $G_2$ commute elementwise, and their intersection is isomorphic to $\left(\mathbb{C}^*\right)^2$.
And all  elements of $Q(m,n,p)$ and $G_3$ normalize both $G_1$ and $G_2$. So the group $G^0=G_1 G_2$ is connected, normal in $G$, and of dimension
\begin{equation*}
\dim G^0=\dim G_1+\dim G_2-2=m^2+n^2+p^2+2r-3.
\end{equation*}
Moreover, we have
\begin{equation*}
\dim G=\dim G^0.
\end{equation*}

For $x\in \textbf{V}(m,n,p|r)$, let $G_x=\{g\in G|g\cdot x=x\}$ be the stabilizer of $x$ in $G$. By Proposition 2.6 of \cite{LBZ-23}, we have the lower bound
\begin{align*}
\dim_x \textbf{V}(m,n,p|r)\geq & \dim G\cdot x\\
                  =  & \dim G-\dim G_x \\
                  =  & \dim G^0-\dim G_x \\
                  = & m^2+n^2+p^2+2r-3-\dim G_x.
\end{align*}
In particular, if the stabilizer $G_x$ is finite, then $\dim G_x=0$ and the orbit $G\cdot x$ has dimension $\dim G^0$. So we have
\begin{equation}\label{eq-lbound}
\dim_x \textbf{V}(m,n,p|r) \geq \dim G^0 =m^2+n^2+p^2+2r-3.
\end{equation}
If $G_x$ is finite, when $(m, n, p, r)=(2,2,2,7)$, by (\ref{eq-lbound}) we have
$$\dim_x \textbf{V}(2,2,2|7)\geq 4+4+4+14-3=23.$$
Thus, there is no gap between upper and lower bounds
for $\dim_x \textbf{V}(2,2,2|7)$. Hence, here for Strassen's algorithm, there is no appearing ``contradiction" with de Groote's theorem \cite{deGr78-2} where the variety of algorithms is of projective dimension nine.  In the sense of de Groote, ``the variety of algorithms" is the set of equivalent classes of $\textbf{V}(m,n,p|r)$ under the actions of $G_2$ and $G_3$ (or $\textbf{V}(m,n,p|r)/G_2G_3$), rather than $\textbf{V}(m,n,p|r)$ itself.

For $x\in\textbf{V}(3,3,3|23)$, when
$G_x$ is finite, by (\ref{eq-lbound}) we have
\begin{align}\label{eq-lb333}
\dim_x \textbf{V}(3,3,3|23) \geq & \dim G\cdot x=\dim G-\dim G_x \notag\\
            =&\dim G^0-0\\
            =& 70 \notag.
\end{align}
Similarly, for $x\in\textbf{V}(4,4,4|49)$, when
$G_x$ is finite we have
\begin{align}\label{eq-lb444}
\dim G\cdot x=\dim G^0-0=143.
\end{align}

\end{remark}

\subsection{Deflation test on $\textbf{V}(3,3,3|23)$}\label{subsec-333}
\

In the website \cite{Heule19}, using SAT solvers Heule et al. provided 17376 solutions in $\textbf{V}(3,3,3|23)$ which are inequivalent under the de Groote group action  \cite{Bur15,deGr78-1}.
We apply the deflation test on these solutions. We found that most of these solutions are ultrasingular. Let $(n_0,n_1,n_2,n_3)$ be the first 4 numbers of the deflation sequence. Overall, from the experimental result, we can see that $76\leq n_0\leq 95$ (see also \cite[Sect. 3]{LBZ-23}). After doing deflation 3 times, we can see that $70\leq n_3\leq 86$. Thus, comparing with (\ref{eq-lb333}),
after doing deflation the gaps between the upper and lower bounds of
  $\dim_x \textbf{V}(3,3,3|23)$ are not large.
Interestingly, among the 17376 solutions there exists a unique solution whose first 4 numbers of the deflation sequence are
\begin{equation}\label{eq-ds333}
(n_0,n_1,n_2,n_3)=(91,73,70,70).
\end{equation}
So the gap between the upper and lower bound is zero for this solution.
The data can be found on \cite{L-23}. The results are summarized below.

\subsubsection{The first step of deflation}
Among the 17376 solutions, after doing the first step of deflation, there are 16511 solutions whose nullities of Jacobians are strictly decreased, that is, $n_0>n_1$. Let $d_1=n_0-n_1$. Then the numerical experiment tells us that $0\leq d_1\leq 23$.
The distribution is as follows.
\begin{table}[h]
\begin{tabular}{|l|l|l|l|l|l|}
\hline
$d_1=n_0-n_1$ & $d_1=0$ &  $1\leq d_1\leq 3$ & $4\leq d_1\leq 10$ & $11\leq d_1\leq 23$ & Total number\\ \hline
Number & 865 & 5495 & 9375  &  1641 &17376\\ \hline
\end{tabular}
\end{table}

\subsubsection{The second step of deflation}

After doing the second step of deflation, there are 16206 solutions whose nullities of Jacobians are strictly decreased, that is, $n_1>n_2$. Let $d_2=n_1-n_2$. Then the numerical experiment tells us that $0\leq d_2\leq 10$.
The distribution is as follows.
\begin{table}[h]
\begin{tabular}{|l|l|l|l|l|}
\hline
$d_2=n_1-n_2$ & $d_2=0$ &  $1\leq d_2\leq 3$ & $4\leq d_2\leq 10$ & Total number\\ \hline
Number & 1168 & 14732 & 1476   &17376\\ \hline
\end{tabular}
\end{table}

\subsubsection{The third step of deflation}

In this step, there are only 978 solutions whose nullities of Jacobians are strictly decreased, that is, $n_2>n_3$. Let $d_3=n_2-n_3$. Then the numerical experiment tells us that $0\leq d_3\leq 5$. By Theorem \ref{thm-nulldi}, we know that $n_2\geq n_3$. However, due to the testing being numerical, some errors appear for 41 solutions, that is, $n_2-n_3<0$. On the other hand, if we restart the deflation process separately for these solutions, we can find that $n_2\geq n_3$. The distribution can be seen
as follows.
\begin{table}[h]
\begin{tabular}{|l|l|l|l|l|l|}
\hline
$d_3=n_2-n_3$ & $d_3=0$ &  $ d_3=1$ & $2\leq d_3\leq 5$ & Error ($d_3<0$) & Total number\\ \hline
Number & 16357 & 795 & 183  &  41 &17376\\ \hline
\end{tabular}
\end{table}

The first 4 numbers of deflation sequence of Laderman's solution \cite{Lad-76} are (76,76,76,76).  The first 4 numbers of deflation sequence of Smirnov's solution \cite{Smir-13} in $\textbf{V}(3,3,3|23)$ are (87,82,79,79).

\subsection{Deflation test on $\textbf{V}(4,4,4|49)$}
\

In \cite{Faw22,Faw22+},  Fawzi et al. provided 14236 solutions in $\textbf{V}(4,4,4|49)$, which are inequivalent under the isotropy group action. We take deflation on these points and find that the ultrasingular solutions are relatively small among all the 14236 solutions.

Note that $B(4,4,4|49)$ has 2352 variables and 4096 equations. For each solution, it takes about 10 minutes for the computation of $(n_0,n_1,n_2,n_3)$. It will spend a lot of time if we take the first 3 steps of deflation for all 14236 solutions. On the other hand, the time for the computation of $(n_0,n_1)$ is relatively acceptable. So we firstly take the first step of  deflation on Fawzi et al.'s solutions and obtain $(n_0,n_1)$ of the 14236 solutions, the data can be found on \cite{L-23}.   The distribution is summarized as follows.
\begin{table}[h]
\begin{tabular}{|l|l|l|l|}
\hline
 & $n_1=n_0=197$ &  $ n_0>n_1$ &  Total number\\ \hline
Number &  13530 & 706  &   14236\\ \hline
\end{tabular}
\end{table}

Secondly, from above table, for the 13530 solutions whose $n_0=n_1=197$, we randomly choose 100 solutions and take the second and third deflation steps to obtain $(n_2,n_3)$.  For all the 100 solutions, we find that $n_2=n_3=197$, so the first 4 numbers of their deflation sequences are all equal to: (197,197,197,197). So this implies that for all 13530 solutions we should have $n_0=n_1=n_2=n_3=197$.

For the remaining 706 solutions, their first 4 numbers of deflation sequences ($n_0,~n_1,~n_2,~n_3$) are computed.  We found that:
\begin{enumerate}
  \item $151\leq n_0\leq 208$;
  \item $ 145\leq n_3\leq 182$;
  \item for most solutions we have $n_2=n_3$.
\end{enumerate}

\subsection{Deflation test on other well-known solutions}
\

Firstly, the first 4 numbers of deflation sequence Smirnov's solution \cite{Smir-13} that lies in $\textbf{V}(3,3,6|40)$ are: (131, 131, 131, 131). So it seems that this solution is semiregular.
Recently, Kauers and Moosbauer found many new solutions of $B(m,n,p|r)$ and put them on Github \cite{KM-23-1,KM-23-2}. We take deflation process for some solutions that are suitable to the complex field.
The first 4 numbers of deflation sequence are summarized in Table \ref{tab-KM}.
\begin{table}[h]
\begin{tabular}{|l|l|}
\hline
$(m,n,p|r) $   & $(n_0,n_1,n_2,n_3)$     \\ \hline
$(3,4,5|47)$   & (152,149,148,148 )      \\ \hline
$(4,4,5|62)$   & (201,194,193,193 )      \\ \hline
$(4,5,5|76)$   & (251,251,251,251 )      \\ \hline
$(5,5,5|97)$   & (395,344,339,339 )      \\ \hline
\end{tabular}
\caption{Deflation sequences for some solutions in \cite{KM-23-1,KM-23-2}}\label{tab-KM}
\vspace{-1.5em}
\end{table}

In Table \ref{tab-KM}, $(m,n,p|r)$ means the solution lie in the set $\textbf{V}(m,n,p|r)$, which can be found in the Github database of the article \cite{KM-23-1,KM-23-2}. In particular, the computation of $(395,344,339,339)$ for $(5,5,5|97)$ spends the time about 4.8 hours on our server.

\subsection{Deflation test on natural algorithms}\label{se-natne}
\

When $r=mnp$, the definition of $\langle m,n,p\rangle$ in (\ref{eq-mmt}) provides a solution of $B(m,n,p|mnp)$ which is denoted by $N(m,n,p)$. Before Strassen's algorithm \cite{Stra-69},
we usually use the algorithm induced in (\ref{eq-mmt}) for the computation of matrix multiplication. So we call $N(m,n,p)$ the \emph{natural algorithm}.

Intuitively, $N(m,n,p)$ should be semiregular. So we should have
$n_0=n_1=n_2=n_3$ for their deflation sequences. When $2\leq m,n,p\leq 4$,  our experimental results are summarized in Table \ref{tab-Nat}. In Table \ref{tab-Nat}, it is interesting to see that there is a big drop between $n_1$ and $n_2$. It tells us that the natural algorithms are geometric singular.

\begin{table}[h]
\begin{tabular}{|l|l|}
\hline
             & $(n_0,n_1,n_2,n_3)$     \\ \hline
$N(2,2,2)$   & (40,40,32,22 )      \\ \hline
$N(2,2,3)$   & (72,72,48,34)      \\ \hline
$N(2,3,3)$   & (126,126,54,50)      \\ \hline
$N(3,3,3)$   & (216,216,72,72)    \\ \hline
$N(3,3,4)$   & (324,324,96,96)     \\ \hline
$N(3,4,4)$   & (480,480,126,126)      \\ \hline
$N(4,4,4)$   & (704,704,164,164)      \\ \hline
\end{tabular}
\caption{Deflation sequences for natural algorithms}\label{tab-Nat}
\vspace{-1.5em}
\end{table}

\section{final remarks and open problems}\label{se-ropen}

\subsection{The local dimension and the inequivalent scheme}
Given a solution $x\in\textbf{V}(m,n,p|r)$ with deflation sequence ($n_0$, $n_1,\ldots,n_{i_0}$), then $n_k$ ($1\leq k\leq i_0$) can be considered as the improved upper bound of the local dimension of $x$.
Suppose that $n_{i_0}=\dim_x \textbf{V}(m,n,p|r)$ and it is strictly greater than the lower bound provided in (\ref{eq-lbound}):
$$
n_{i_0}>m^2+n^2+p^2+2r-3.
$$
That is, the local dimension of $x$ is strictly larger than the dimension of the group orbit $G\cdot x$, where $G$ is the group studied in Remark \ref{rem-2227}. A lower dimensional algebraic set is of measure zero, when we put it in a higher dimensional algebraic set. So in the neighbourhood of $x$, there exist many (in fact, infinite many) other different $G$-orbits.
Since the isotropy group orbit is contained in the $G$-orbit, there exist many other different isotropy group orbits in the neighbourhood of $x$. That is, in the neighbourhood of $x$, there are many other $y\in \textbf{V}(m,n,p|r)$ such that $x$ and $y$ are not lie in the same isotropy group orbit. In other words, $x$ and $y$ are inequivalent matrix multiplication schemes \cite{Heule21}.



\subsection{The local dimension and the deflation sequence}

Let $\hat{\textbf{x}}$ be a solution of a polynomial system $F$. Suppose that $\hat{\textbf{x}}$ lies on a a $k$-dimensional ultrasingular irreducible component. Let
$$(n_0,n_1,\ldots,n_{i_0})$$
be  deflation sequence of $\hat{\textbf{x}}$  defined in (\ref{eq-defsequence}). It is an open problem to give an estimation of the number $i_0$. In particular, for Brent equations, can we give some estimations?

From Section \ref{subsec-333}, we know that for the 17376 solutions in $\textbf{V}(3,3,3|23)$, most of their deflation sequences satisfy
$$n_0>n_1>n_2=n_3.$$
For other solutions studied in Section \ref{se-numerexp}, we can see that most of their deflation sequences satisfy
$$n_2=n_3,$$
that is, they begin to stabilize after doing deflation 2 times.
In the deflation sequences (\ref{eq-ds222}) and (\ref{eq-ds333}), we know that $n_2=n_3$ and they are the local dimensions of the corresponding solutions. So combining the numerical experimental results in Section \ref{se-numerexp}, we have the following question.
\begin{question}
Suppose that $x\in\textbf{V}(m,n,p|r)$ with deflation sequence $(n_0,n_1,\ldots,n_{i_0})$ defined in Definition \ref{def-defseq}.
Give an upper bound of $i_0$. Moreover, we wonder if $i_0\leq3$?
\end{question}


%



\subsection{The deflation sequence as an invariant}
In \cite[Quest. 3.15]{LBZ-23}, we observed that the rank of Jacobian of $B(m,n,p|r)$ may be served as an invariant to decide the equivalent classes under the de Groote group action  \cite{Bur15,deGr78-1}. Further discussions of de Groote group action can be seen from \cite{Heule21,Lands-17,Tichav-21}, etc. In this paper, from the numerical experiment, we also find that the deflation sequences are unchanged when the solutions are
transformed by the de Groote group action. So we have the following question.
\begin{question}
For two points in $\textbf{V}(m,n,p|r)$, if their deflation sequences are different, do they also lie in different group orbits under the de Groote group action?
\end{question}

For the deflation sequence $(n_0,n_1,\dots)$, $n_0$ is the nullity of Jacobian. So the deflation sequence contains more information and is more useful to distinguish the equivalent classes.  For the solutions in \cite{Heule19}, in our numerical experiment, it is easy to find two deflation sequences $(n_0,n_1,\dots)$ and $(n_0',n_1',\dots)$ such that $n_0=n_0'$ but $n_1\neq n_1'$.

However, for the solutions provided in \cite{Faw22+}, almost all solutions have the deflation sequence:(197,197,197,197).  So in this case, distinguishing equivalent classes by deflation sequences is not so effective. In \cite[Sect. 3.3]{LBZ-23}, we discuss the gap between the upper and lower bound, and ask when $r$ is the tensor rank if the gap is zero (or very small). From (\ref{eq-lb444}), the gap between the upper and lower bound is $197-143=54$, which is relatively large. Moreover, for $\langle4,4,4\rangle$ the tensor decomposition of length 47 over $\mathbb{Z}_2$ has been found in \cite{Faw22}.
So this implies that the tensor rank of $\langle4,4,4\rangle$ may be less than 49.

\subsection{Deciding the singularity of a solution}\label{subs-desin}
Given a polynomial system and a solution, which type of the solution belongs to: numerical smooth/singular or geometric smooth/singular? The deflation sequence only tells us partial information about the type.

\subsection{The real deflation process}
By now, the deflation processes in \cite{DZeng-05,DLZeng-11,HW-13, LVZ-06} are given for the complex algebraic set. In practice, the polynomial system we are interested is often of real coefficients and has real solutions. For example, the Brent equations are of this type.
Are the results in \cite{DZeng-05,DLZeng-11,HW-13, LVZ-06} also suitable to the real algebraic set? For example, it is not hard to see that Conjecture \ref{conj-dc} also has a real counterpart.

\subsection{On local real dimension and local complex dimension}
From the relation between fast matrix algorithms and solutions of the Brent equations \cite{Bur15,Heule21}, we can see that integer, rational and real solutions are more useful than complex solutions.
For real solutions of the Brent equation, we want to know the difference between the local real dimension and local complex dimension. Related discussions are given in the paper \cite{HHS-23}.


\subsection{On the numerical determination of local dimension}
On how to determine the local dimension, Kuo and Li provided a numerical method in \cite{KL-08}. In addition, in the neighbourhood of the solution, the convergence behavior under the rank-$r$ projection iteration in \cite{Zeng-23} can also help us identify the local dimension of the solution.


\section*{Acknowledgments}
We are very grateful to the editors and referees for their valuable comments and suggestions. In particular, we are grateful to the anonymous referee who kindly reminds us the linear group actions
we have overlooked, which improves the Remark \ref{rem-2227}.

Xin Li is supported by National Natural Science Foundation of China (Grant No.11801506).  Liping Zhang is supported by the Zhejiang Provincial Natural Science Foundation of China (Grant No. LY21A010015) and Grant 11601484, the National Natural Science Foundation, People's Republic of China.

The discussion with Professor Zhonggang Zeng of Northeastern Illinois University inspired us a lot. We are grateful to him for so many valuable comments and suggestions. We are grateful to Yuan Feng, Xiaodong Ding, Liaoyuan Zeng, Changfeng Ma and Jinyan Fan for their kind supports and suggestions. Many thanks to Professor
J. Hauenstein, J. Landsberg, Petr Tichavsk\'{y} and Ke Ye for many helpful discussions when we prepare this paper.
Xin Li is grateful to Yaqiong Zhang for her consistent support and understanding.

\bibliographystyle{amsplain}

\end{document}